\documentclass[leqno]{amsart}
\usepackage{url}

\newcommand{\C}{{\mathbb{C}}}
\newcommand{\F}{{\mathbb{F}}}

\newcommand{\Q}{{\mathbb{Q}}}

\newcommand{\R}{{\mathbb{R}}}
\newcommand{\Z}{{\mathbb{Z}}}

\newcommand{\longhookrightarrow}{\lhook\joinrel\longrightarrow}

\newtheorem{theorem}{Theorem}[section]

\newtheorem{corollary}[theorem]{Corollary}

\newtheorem{proposition}[theorem]{Proposition}\theoremstyle{definition}  
\newtheorem{proposition/definition}[theorem]{Proposition/Definition}\theoremstyle{definition}

\numberwithin{theorem}{section}
\numberwithin{equation}{section}

\begin{document}

\title{Stickelberger and the Eigenvalue Theorem}

\author{David A. Cox}
\address{Department of Mathematics and Statistics, Amherst College, Amherst, MA 01002}
\email{dacox@amherst.edu}

\begin{abstract}
This paper explores the relation between the Eigenvalue Theorem and the work of Ludwig Stickelberger (1850-1936).
\end{abstract}

\dedicatory{To David Eisenbud on the occasion of his 75th birthday.}

\subjclass[2010]{11R29, 13P15, 15A18}

\keywords{eigenvalue, trace, quadratic form, polynomial system}
\maketitle

\section{Introduction} The Eigenvalue Theorem is a standard result in computational algebraic geometry.  Given a field $F$ and polynomials $f_1,\dots,f_s \in F[x_1,\dots,x_n]$, it is well known that the system
\begin{equation}
\label{syspolys}
f_1 = \cdots = f_s = 0
\end{equation}
has finitely many solutions over the algebraic closure $\overline{F}$ of $F$ if and only if 
\[
A = F[x_1,\dots,x_n]/\langle f_1,\dots,f_s\rangle
\]
has finite dimension over $F$ (see, for example, Theorem 6 of \cite[Ch.~5, \S3]{CLO}).  

A polynomial $f \in F[x_1,\dots,x_n]$ gives a
multiplication map
\[
m_f : A \longrightarrow A.
\]
A basic version of the Eigenvalue Theorem goes as follows:

\begin{theorem}[Eigenvalue Theorem]
\label{ETbasic}
\label{eigenvalue}
When $\dim_F A < \infty$, the eigenvalues of $m_f$ are the values of $f$ at the finitely many solutions of \eqref{syspolys} over $\overline{F}$.
\end{theorem}

For $\overline{A} = A \otimes_F \overline{F}$, we have a canonical isomorphism of $F$-algebras
\[
\overline{A} =  \prod_{a \in  \mathbf{V}_{\overline{F}}(f_1,\dots,f_s)} \overline{A}_a,
\]
where $\overline{A}_a$ is the localization of $\overline{A}$ at the maximal ideal corresponding to $a$.  Following \cite[Thm.\ 3.3]{tapas}, we get a more precise version of the Eigenvalue Theorem:

\begin{theorem}[Stickelberger's Theorem]
\label{ETStick}
For every $a \in \mathbf{V}_{\overline{F}}(f_1,\dots,f_s)$, we have $m_f(\overline{A}_a) \subseteq \overline{A}_a$, and the restriction of $m_f$ to $\overline{A}_a$ has only one eigenvalue $f(a)$.  
\end{theorem}

This result easily implies Theorem~\ref{ETbasic} and enables us to compute 
the characteristic polynomial of $m_f$. Namely, the multiplicity of $a$ as a solution of \eqref{syspolys} is
\[
\mu(a) = \dim_{\overline{F}} \overline{A}_a,
\]
and then Theorem~\ref{ETStick} tells us that the characteristic polynomial of $m_f$ is
\begin{equation}
\label{mfchpoly}
\det(m_f-x\hskip1pt I) = \!\!\prod_{a \in  \mathbf{V}_{\overline{F}}(f_1,\dots,f_s)} \!\!(f(a) -x)^{\mu(a)}.
\end{equation}
Furthermore, since the trace of a matrix can be read off from its characteristic polynomial, \eqref{mfchpoly} gives the formula
\begin{equation}
\label{mftrace}
\mathrm{Tr}(m_f) = \!\!\sum_{a \in  \mathbf{V}_{\overline{F}}(f_1,\dots,f_s)} \!\!\mu(a) f(a).
\end{equation} 
This trace formula will play an important role in what follows.

The name ``Stickelberger's Theorem'' in Theorem~\ref{ETStick} is from \cite{tapas}.  Versions of Theorems~\ref{ETbasic} and~\ref{ETStick} also named ``Stickelberger's Theorem'' can be found in the papers \cite{GPV,sottile,Zeng}, and \cite{Scheib} has a ``Stickelberger's Theorem'' for 
positive-dimensional solution sets.  A ``Stickelberger's Theorem'' that focuses on \eqref{mfchpoly} and \eqref{mftrace} can be found in \cite{BPR}. A common feature of these papers is that no reference to Stickelberger is given!  An exception is \cite{GPV}, which refers to the wrong paper of Stickelberger.  

There is an actual theorem of Ludwig Stickelberger lurking in the background, in the paper \emph{\"Uber eine neue Eigenschaft der Diskriminanten algebraischer Zahlk\"orper} \cite{S1897} that appeared in the proceedings of the first International Congress of Mathematicians, held in Z\"urich in 1897.  This paper includes Theorems I--XIII, most dealing with traces and properties of the discriminant of a number field.  

In \cite{S1897}, Stickelberger fixes a number field $\Omega$ of degree $n$ and discriminant $D$.  Here are two of the theorems from \cite{S1897}:

\begin{theorem}[Theorems VII and XIII of \cite{S1897}]
\label{Stick2Thm}
If a prime $p$ does not divide $D$, then the Legendre symbol $\big(\tfrac{D}{p}\big)$ satsifes
\[
\Big(\frac{D}{p}\Big) = (-1)^{n-m},
\]
where $p\mathcal{O} = \mathfrak{p}_1\cdots \mathfrak{p}_m$ is the prime factorization in the ring $\mathcal{O}$ of algebraic integers of $\Omega$.
\end{theorem}

This result is well known in number theory.  See, for example, \cite{Carlitz} and  
\cite{KK}.  But for our purposes, Stickelberger's most interesting theorem in \cite{S1897} involves the trace function of $ \mathcal{O}$ modulo an ideal $\mathfrak{a}$ containing a prime $p$. This is the map
\[
\mathrm{Tr}_\mathfrak{a}: \mathcal{O} \longrightarrow \F_p
\]
where multiplication by $\alpha \in \mathcal{O}$ gives a $\F_p$-linear map $m_\alpha: \mathcal{O}/\mathfrak{a} \to \mathcal{O}/\mathfrak{a}$ with trace
\[
\mathrm{Tr}_\mathfrak{a}(\alpha) = \mathrm{Tr}(m_\alpha) \in \F_p.
\]
When $\mathfrak{a} = p\mathcal{O}$, we write $\mathrm{Tr}_p(\alpha)$ instead of $\mathrm{Tr}_{p\mathcal{O}}(\alpha)$.  Here is Stickelberger's theorem:

\begin{theorem}[Theorem III of \cite{S1897}]
\label{StickTrace}
Let $p$ be prime with factorization $p\mathcal{O} = \mathfrak{p}^{e_1}_1\cdots \mathfrak{p}_m^{e_m}$, where $\mathfrak{p}_1,\dots, \mathfrak{p}_m$ are distinct primes.  Then for any $\alpha \in \mathcal{O}$, we have
\[
\mathrm{Tr}_p(\alpha) = \sum_{i=1}^m e_i \mathrm{Tr}_{\mathfrak{p}_i}(\alpha).
\]
\end{theorem}

Given the similarity to the trace formula \eqref{mftrace}, it becomes clear why Stickelberger's paper is relevant to the Eigenvalue Theorem.  The link was made explicit in 1988 when G\"unter Scheja and Uwe Storch published \emph{Lehrbuch der Algebra} \cite{SS}.  However, even though Scheja and Storch invoke Stickelberger's name, they do not refer to his 1897 paper \cite{S1897}.

In what follows, we will say more about Stickelberger and his mathematics in Section 2 and explore the history of the Eigenvalue Theorem in Section 3.  Section 4 will describe how Stickelberger and the Eigenvalue Theorem came together in 1988 under the influence of Scheja and Storch, and Section 5 will explain the unexpected role played by real solutions.  We end with some final remarks in Section~6

\section{Ludwig Stickelberger}

Ludwig Stickelberger was a Swiss mathematician born in 1850 in the canton of Schaffhausen and died in 1936 in Basel.  He got his PhD from Berlin in 1874 under the direction of Ernst Kummer and Karl Weierstrass.  After spending a few years at the forerunner of ETH in Z\"urich, Stickelberger went to the University of Freiburg in 1879.  He retired in 1919 but remained in Freiburg as an ``Honorarprofessor'' until 1924, when he returned to Switzerland.

Stickelberger's mathematical work is described in a 1937 article \cite{Hef} written by his Freiburg colleague Lothar Heffter.  Stickelberger's mathematical output was modest:\ besides his dissertation, he published 12 papers during his lifetime, four jointly written with Frobenius.  One unpublished manuscript from 1915 appeared posthumously in 1936.  Heffter gives a brief description of each paper in \cite{Hef}.

His papers cover a range of topics, including quadratic forms, real orthogonal transformations, differential equations, algebraic geometry, group theory, elliptic functions, and algebraic number theory.  Heffter comments that 
\begin{quote}
\dots\ he definitely adopted Gauss' point of view ``Pauca sed matura" [few but mature]. He recognized and filled essential gaps in fundamental theories, often having the last word with the keystone of a development that gives the theory its final, simplest form.
\end{quote}

Stickelberger's best known result, published in 1890 in \emph{Mathematische Annalen} \cite{S1890}, concerns an element $\theta$ in the group ring $\Q[G]$, where $G$ is the Galois group $\mathrm{Gal}(\Q(\zeta_m)/\Q) \simeq (\Z/m\Z)^\times$ of the cyclotomic extension $\Q \subseteq \Q(\zeta_m)$.  This gives the ideal
\[
I = (\theta \Z[G])\cap \Z[G] \subseteq \Z[G].
\]
It is customary to call $\theta$ the \emph{Stickelberger element} and $I$ the \emph{Stickelberger ideal}.  Here is his theorem:

\begin{theorem}[Stickelberger's Theorem \cite{S1890}]
\label{STThm}
The Stickelberger ideal $I$ annihilates the class group of $\Q(\zeta_m)$.
\end{theorem}

If you search MathSciNet for reviews that mention ``Stickelberger'' anywhere, the vast majority involve the Stickelberger element, the Stickelberger ideal, and their generalizations.   When mathematicans say ``Stickelberger's Theorem'', they are usually referring to Theorem~\ref{STThm}.  This is probably what led the authors of \cite{GPV} to cite \cite{S1890} as the source for their version of the Eigenvalue Theorem.

In 1897, Stickelberger published the paper \cite{S1897} discussed in the Introduction.  The main focus here is on properties of the discriminant $D$ of a number field $\Omega$.  Besides proving Theorem~\ref{Stick2Thm}, Stickelberger's results also imply that $D \equiv 0,1 \bmod 4$.  This standard fact appears in many textbooks on algebraic number theory (see, for example, Exericse 7 on p.\ 15 of \cite{Neukirch}, where the congruence is called \emph{Stickelberger's discriminant relation}).  

There are several ways to define $D$; the one most relevant to us uses a $\Z$-basis $\beta_1,\dots,\beta_n$ of the ring $\mathcal{O}$ of algebraic integers of $\Omega$.  The trace function $\mathrm{Tr} : \Omega \to \Q$ maps $\mathcal{O}$ to $\Z$.  Then the discriminant of $\Omega$ is defined to be
\[
D = \det(\mathrm{Tr}(\beta_i\beta_j)) \in \Z.
\]
Given this definition, it is not surprising that Stickelberger begins \cite{S1897} with some properties of traces.  He quickly gets to the trace formula given in Theorem~\ref{StickTrace}, which we propose calling the \emph{Stickelberger Trace Formula} to distinguish it from the more famous Stickelberger Theorem \ref{STThm}.  

In Section 4, we will explain carefully how the Stickelberger Trace Formula relates to the Eigenvalue Theorem.  But first, we need to learn more about the evolution of the Eigenvalue Theorem.

\section{The Eigenvalue Theorem}

A key feature of the Eigenvalue Theorem is that the quotient algebra $A = F[x_1,\dots,x_n]/\langle f_1,\dots,f_s\rangle$ is finite dimensional over $F$ when $f_1 = \cdots = f_s = 0$ has finitely many solutions over $\overline{F}$.  This was known by the end of the 1970s and is what allows us to use linear algebra to find solutions.  But getting from here to the Eigenvalue Theorems \ref{ETbasic} and \ref{ETStick} involved several independent discoveries, each with its own point of view.  In what follows, I will mention some but not all of the relevant papers. 

We begin in 1981 with Daniel Lazard's paper \emph{R\'esolutions des syst\`emes d'\'equations alg\'ebriques} \cite{Laz}, which gives an algorithm to solve a zero-dimensional system.  To relate his approach to ours, observe that setting $x = 0$ in \eqref{mfchpoly} gives the formula
\begin{equation}
\label{detmf}
\det(m_f) = \!\!\prod_{a \in  \mathbf{V}_{\overline{F}}(f_1,\dots,f_s)} \!\!f(a)^{\mu(a)}.
\end{equation}
For new variables $U_0,\dots,U_n$, let $L = U_0 + U_1 x_1 + \cdots + U_n x_n$.  Given a point $a = (a_1,\dots,a_n) \in \overline{F}^n$, applying $L$ to $a$ gives
\[
L(a) = U_0 + U_1 a_1 + \cdots + U_n a_n,
\]
from which we can recover $a$.   Thus, if we could somehow set $f = L$ in \eqref{detmf}, we would get
\begin{equation}
\label{lazard81}
\det(m_L) = \!\!\prod_{a \in  \mathbf{V}_{\overline{F}}(f_1,\dots,f_s)} \!\!L(a)^{\mu(a)},
\end{equation}
which would give the solutions and their multiplicities.

 In \cite{Laz}, Lazard describes an algorithm for computing a projective version of the right-hand side of \eqref{lazard81}.  He replaces $A$ with 
 \[
A_U = F[U_0,\dots,U_n,x_0,\dots,x_n]/\langle F_1,\dots,F_s\rangle,
 \]
where $F_i(x_0,\dots,x_n)$ is the homogenization of $f_i(x_1,\dots,x_n)$ and $L$ becomes $L =  U_0 x_0 + \cdots + U_n x_n$.  The ring $A_U$ is graded with respect to $x_0,\dots,x_n$, and multiplication by $L$ between graded pieces of $A_U$ appears explicitly in \S4 of \cite{Laz}.

The product in \eqref{lazard81} is an example of a $U$-resultant, and (very large) determinantal formulas for such resultants were known by the early 20th century.  Lazard's paper is important because of its efficient algorithm for computing this product.  For us, the key feature of \cite{Laz} is the use of a multiplication map on a quotient algebra.

The next advance came in 1988 with the paper \emph{An elimination
algorithm for the computation of all zeros of a system of multivariate polynomial equations} by Winfried Auzinger and Hans Stetter \cite{AS}.  For a system of $n$ equations in $x_1,\dots,x_n$, their initial is goal is to compute 
the right-hand side of \eqref{detmf} when $f = b_0 + b_1x_1 + \cdots + b_nx_n$.  Coming from a background in numerical analysis, they begin with the classical theory of resultants and describe an approach that works ``in the general case (without degeneracies)''.

In \S5 of \cite{AS},  Auzinger and Stetter construct matrices $B^{(k)}$, $k = 1,\dots,n$, whose eigenvalues are the $k$th coordinates of the solutions, together with simultaneous eigenvectors.  They also explain how these eigenvectors enable one to find the solutions.  Eigenvalues and eigenvectors finally take center stage!

For us, \S6  of \cite{AS} is the most interesting, for here, $B^{(k)}$ is interpreted as the matrix of the linear map $x_k : A \to A$ given by multiplication by $x_k$.   Then comes a key observation: while the treatment so far assumes that there are no degeneracies, one can avoid this assumption by  simply \emph{defining} $B^{(k)}$ to be the matrix of  multiplication by  $x_k$ on $A$.  Everything still works and we finally have the Eigenvalue Theorem!

A more complete treatment of this circle of ideas appears in the \emph{Central Theorem} (Theorem 2.27) in Stetter's 2004 book \emph{Numerical Polynomial Algebra} \cite{Stet}.  You can also read about this in \emph{Using Algebraic Geometry} \cite{UAG}, where \S2.4 discusses the Eigenvalue Theorem and the role of eigenvectors, and \S3.6 makes the link to resultants when there are no degeneracies.  We should also mention the 1992 paper  \emph{Solutions of systems of algebraic equations and linear maps on residue class rings} \cite{YNT} by Yokoyama, Noro and Takeshima that draws on ideas of Lazard,  Auzinger and Stetter, together with papers of Kobayashi.

In the Historical and Bibliographical Notes to Chapter 2 of \cite{Stet}, Stetter writes
\begin{quote}
The fundamental relation between the eigenelements of multiplication in the quotient ring and the zeros of the ideal must have been known to algebraists of the late 19th and early 20th centuries, in the language of the time.  \dots\  There are quotations of a theorem of Stickelberger from the 1920s, which is equivalent to Theorem 2.27, but its relevance remained concealed.  
\end{quote}
Sorting out what was known 100 years ago is not an easy task.  The only name mentioned by Stetter is our friend Stickelberger, though as we have seen, the date is 1897, not the 1920s.  

It is now time to turn to Stickelberger, even though the above discussion omits some important papers from the early 1990s that are relevant to the ideas behind the Eigenvalue Theorem.  We will consider this work in \S5 when we study real solutions of a polynomial system.  

\section{Scheja and Storch 1988}

In 1988, G\"unter Scheja and Uwe Storch published the two-volume algebra text \emph{Lehrbuch der Algebra}.  In Volume 2, \S94 deals with trace forms (Spurformen) and is where Stickelberger enters the picture:
\begin{quote}
{\bf Beispiel 7} (Die S\"atze von Stickelberger)
\end{quote}
(see \cite[p.\ 795]{SS}).  But before giving the theorems, they observe that
\begin{quote}
In some cases, the fine structure of the trace form of a finite free algebra can be described with the help of simple features of the algebra itself. 
\end{quote}
They begin with a ``simple lemma" that goes as follows.  Let $A$ be a finite-dimensional $F$-algebra with maximal ideals $\mathfrak{m}_1,\dots, \mathfrak{m}_r$.  The localizations $A_{\mathfrak{m}_i}$ have residue fields $L_i \simeq A/\mathfrak{m}_i$ and satisfy
\[
A \simeq \prod_{i=1}^r A_{\mathfrak{m}_i}.
\]
For each $i$, define $\lambda_i$ by the equation
\begin{equation}
\label{lambdaidef}
\dim_F A_{\mathfrak{m}_i} = \lambda_i [L_i : F]
\end{equation}
Note also that $\alpha \in A$ gives $F$-linear multiplication maps $m_\alpha : A \to A$ and $m_\alpha : L_i \to L_i$.  

\begin{theorem}[Lemma 94.6 in \cite{SS}]
\label{ETSS}
Assume that $L_i$ is a separable extension of $F$ for $1 \le i \le r$.  Then for $\alpha \in A$, the multiplication maps $m_\alpha$ defined above satisfy
\[
\mathrm{Tr}_A(m_\alpha) = \sum_{i=1}^r \lambda_i \mathrm{Tr}_{L_i}(m_\alpha).
\]
\end{theorem}

\begin{proof}
Since  $A \simeq A_{\mathfrak{m}_1} \times \cdots \times A_{\mathfrak{m}_r}$, we can reduce to the case where $A$ is local with maximal ideal $\mathfrak{m}$ and residue field $L = A/\mathfrak{m}$. Then \eqref{lambdaidef} can be written
\[
\dim_F A = \lambda \hskip1pt [L:F].
\]
In the $F$-algebra $A$, the separable hull $A_{\mathrm{sep}} \subseteq A$ consists of all elements of $A$ whose minimal polynomial over $F$ is separable.  Since $F \subseteq L$ is separable by hypothesis, the composition
\[
A_{\mathrm{sep}} \longhookrightarrow A \longrightarrow A/\mathfrak{m} = L
\]
is an isomorphism of fields by Corollary 91.14 of \cite{SS}.  

Thus $A$ is a vector space over $A_{\mathrm{sep}}$ of dimension $\lambda$.
 A basis $\{\beta_1,\dots,\beta_\lambda\}$ of $A$ over $A_{\mathrm{sep}}$ gives a direct sum
\[
A = A_{\mathrm{sep}}\beta_1 \oplus \cdots \oplus A_{\mathrm{sep}}\beta_\lambda.
\]
To compute the trace of $m_\alpha : A \to A$ over $F$, first assume $\alpha \in A_{\mathrm{sep}}$.  Then $m_\alpha$ is compatible with the direct sum  decomposition, which easily implies 
\[
\mathrm{Tr}_{A}(m_\alpha) = \lambda \mathrm{Tr}_{A_{\mathrm{sep}}}(m_\alpha).
\]
Via the isomorphism $A_{\mathrm{sep}} \simeq L$, this becomes
\begin{equation}
\label{lambdaidef2}
\mathrm{Tr}_{A}(m_\alpha) = \lambda \mathrm{Tr}_{L}(m_\alpha).
\end{equation}

Now suppose $\alpha \in A$ is arbitrary.  Since  $A_{\mathrm{sep}} \simeq L = A/\mathfrak{m}$, there is $\alpha' \in A_{\mathrm{sep}}$ such that $\alpha = \alpha' + \beta$ with $\beta \in \mathfrak{m}$.  Note that $\mathfrak{m}$ is nilpotent since $A$ is finite-dimensional over $F$.  Then $m_\alpha = m_{\alpha'} + m_\beta$ implies
\begin{align*}
\mathrm{Tr}_{A}(m_\alpha) &= \mathrm{Tr}_{A}(m_{\alpha'}) + \mathrm{Tr}_{A}(m_\beta)\\ &= \lambda \mathrm{Tr}_{L}(m_{\alpha'}) + 0\\
&= \lambda \mathrm{Tr}_{L}(m_{\alpha'}) + \lambda \mathrm{Tr}_{L}(m_\beta) = \lambda \mathrm{Tr}_{L}(m_\alpha),
\end{align*}
where the second line follows from \eqref{lambdaidef2} with $\alpha'$ and the fact that $m_\beta$ is nilpotent, and the third line follows since $m_\beta$ is the zero map on $L$.  We conclude that \eqref{lambdaidef2} holds for all $\alpha \in A$, and the theorem follows.
\end{proof}

A key feature of the proof is that everything becomes clear once we understand the structure of $A$, i.e., the ``simple features of the algebra itself".

Scheja and Storch then use Theorem \ref{ETSS} to prove various results of Stickelberger, including the Stickelberger Trace Formula given in Theorem~\ref{StickTrace}.  For this reason, it makes sense to call Theorem \ref{ETSS} the \emph{Stickelberger Trace Formula} in honor of Stickelberger's contribution.

Here is an easy consequence of Theorem \ref{ETSS}.

\begin{corollary}
\label{StickTrThm}
With the notation and assumptions of Theorems \ref{ETbasic} and \ref{ETStick}, we have
\[
\mathrm{Tr}(m_f) = \!\!\sum_{a \in  \mathbf{V}_{\overline{F}}(f_1,\dots,f_s)} \!\!\mu(a) f(a).
\]
\end{corollary}

\begin{proof} Maximal ideals  $\mathfrak{m}_1,\dots, \mathfrak{m}_r$ of $\overline{A} = A \otimes_F \overline{F}$ 
correspond to solutions $a_1,\dots,a_r$ in $\mathbf{V}_{\overline{F}}(f_1,\dots,f_s)$.  In 
the notation of Section 1, the localization $\overline{A}_{a_i}$ has residue field
$L_i \simeq \overline{F}$ via the map $\overline{A}_{a_i} \to \overline{F}$ defined by $f \mapsto f(a_i)$.  It follows that $\mathrm{Tr}_{L_i}(m_f) = f(a_i)$.  Furthermore, for $\overline{A}_{a_i}$, the $\lambda_i$ in \eqref{lambdaidef} is multiplicity $\mu(a_i)$.  Then the desired formula for $\mathrm{Tr}(m_f)$ is an immediate consequence of Theorem \ref{ETSS}.
\end{proof}

Corollary~\ref{StickTrThm} is the trace formula \eqref{mftrace} from the Introduction.  There, we deduced  \eqref{mftrace} from the version of ``Stickelberger's Theorem'' given in \cite{tapas}.  We now see how this follows from Theorem~\ref{ETSS}, which is Scheja and Strorch's version of the actual Stickelberger Trace Formula from 1897.

This is nice, but where are the eigenvalues?  After all, our main concern is the relation between Stickelberger and the Eigenvalue Theorem.  Fortunately, the trace formula given in Corollary~\ref{StickTrThm} is powerful enough to determine the eigenvalues of $m_f$ when $F$ has characteristic zero.  Here is the precise result:

\begin{proposition}
\label{TrET}
With the notation and assumptions of Theorems \ref{ETbasic} and \ref{ETStick}, the following are equivalent when  $\mathrm{char}(F) = 0${\rm:}
\begin{enumerate}
\item For every $f \in F[x_1,\dots,x_n]$, 
\[
\mathrm{Tr}(m_f) = \!\!\sum_{a \in  \mathbf{V}_{\overline{F}}(f_1,\dots,f_s)} \!\!\mu(a) f(a).
\]
\item For every $f \in F[x_1,\dots,x_n]$, 
\[
\det(m_f-x\hskip1pt I) = \!\!\prod_{a \in  \mathbf{V}_{\overline{F}}(f_1,\dots,f_s)} \!\!(f(a) -x)^{\mu(a)}.
\]
\end{enumerate}
\end{proposition}

\begin{proof} We proved (2) $\Rightarrow$ (1) in the discussion leading up to \eqref{mftrace}.  As for (1) $\Rightarrow$ (2), let $M$ be the diagonal matrix whose diagonal entries are $f(a)$ repeated $\mu(a)$ times, for each $a \in  \mathbf{V}_{\overline{F}}(f_1,\dots,f_s)$.  Then for any integer $\ell \ge 0$, we have
\[
\mathrm{Tr}(M^\ell) = \!\!\sum_{a \in  \mathbf{V}_{\overline{F}}(f_1,\dots,f_s)} \!\!\mu(a) f(a)^\ell = \mathrm{Tr}(m_{f^\ell}) = \mathrm{Tr}((m_f)^\ell),
\]
where the second equality uses (1) with $f^\ell$ and the third equality follows from $m_{fg} = m_f \circ m_g$.  Thus $M^\ell$ and $(m_f)^\ell$ have the same trace for all $\ell \ge 0$.  

It has been known since 1840 that in characteristic zero, the characteristic polynomial of a matrix is determined by the traces of its powers (a formula for the coefficients in terms of the traces is given in \cite{Lewin})\footnote{If $\mathrm{char}(F) = p > 0$, then $A = \lambda I_p$ has $\mathrm{Tr}(A^\ell) = 0$ for all $\ell \ge 0$, independent of $\lambda$, while $\det(A-x\hskip1pt I) = (\lambda - x)^p = \lambda^p - x^p$.}.   Thus the previous paragraph implies that $M$ and $m_f$ have the same characteristic polynomial, and (2) follows immediately.
\end{proof}

We now have a direct path from Stickelberger to the characteristic zero version of the Eigenvalue Theorem.  Our final task is to explore how Stickelberger's name began to appear in the literature following the 1988 publication of Scheja and Storch's book \cite{SS}.  

In Theorem~\ref{ETStick},  we stated ``Stickelberger's Theorem'' from the 1999 book chapter \cite{tapas} by Gonzalez-Vega, Roullier and Roy.  Their Corollary 3.6 states the formulas for the trace, determinant and characteristic polynomial of $m_f$ given in \eqref{mftrace}, \eqref{detmf} and \eqref{mfchpoly} respectively.  Not surprisingly, there is no reference to Stickelberger.  Nor is there a reference to Scheja and Storch!  

However, there is a reference to the 1995 paper \cite{GVT} by Gonzalez-Vega and Trujillo, which includes the following result (reproduced verbatim):
 
\medskip

\noindent {\bf Theorem 1. (Stickelberger Theorem)} \emph{Let $\mathbb{K} \subset \F$ be a field extension with $\F$ algebraically closed, $h \in \mathbb{K}[\underline{x}]$ and $J$ be a zero dimensional ideal in $ \mathbb{K}[\underline{x}]$.  If $\mathcal{V}_\F(J) = \{\varDelta_1,\dots,\varDelta_s\}$ are the zeros in $\F^n$ of $J$ then there exists a basis of  $\F[\underline{x}]/J$ such that the matrix of $\mathrm{M}_h$, with respect to this basis, has the following block structure:
\[
\begin{pmatrix} \mathbb{H}_1 & 0 & \dots & 0\\
 0 & \mathbb{H}_2  & \dots & 0\\ \vdots & \vdots && \vdots\\
0 & 0 & \dots &  \mathbb{H}_s \end{pmatrix} \qquad \text{where} \qquad \mathbb{H}_i = \begin{pmatrix} h(\varDelta_i) & \star & \dots & \star\\
 0 & h(\varDelta_i)  & \dots & \star\\ \vdots & \vdots && \vdots\\
0 & 0 & \dots &  h(\varDelta_i) \end{pmatrix}
\]
The dimension of the $i$-th submatrix is equal to the multiplicity of $\varDelta_i$ as a zero of the ideal $J$.}
\medskip

As far as I know, this is the first explicit mention of ``Stickelberger's Theorem'' in the literature.  As usual, Stickelberger does not appear in the references to \cite{GVT}, and there is also no reference to Scheja and Storch.  To see why, we look to Trujillo's 1997 PhD thesis \cite{trujillo}. She states a version of  Theorem 1 and says:
 \begin{quote}
The version presented here was introduced by L. Stickelberger ([SS88]) in 1930
\end{quote}
So we have a direct link between ``Stickelberger's Theorem'' and Scheja and Storch, though the date 1930 is not correct.

Trujillo also notes that this result was rediscovered independently in 1991 by Pedersen, Roy and Szpirglas (see \cite{PRS}) and by Becker and W\"ormann (see \cite{BW1}).  The references to \cite{tapas} and \cite{GVT} cite these authors.  Hence we need to examine  \cite{PRS} and \cite{BW1}.   These papers deal with solutions over $\R$, which leads to our next topic. 

\section{Counting Real Solutions}

Given a finite-dimensional $F$-algebra $A$, multiplication by $\alpha \in A$ gives a $F$-linear map $m_\alpha : A \to A$ as usual.  In \S94 of \cite{SS}, Scheja and Storch define the \emph{trace form} to be the symmetric bilinear form $T_A$ on $A$ defined by
\[
T_A(\alpha,\beta) = \mathrm{Tr}(m_{\alpha\beta}) \in F.
\]
When $F = \R$, the \emph{type} of $T_A$ is $(p,q)$, where $p = \#$ positive eigenvalues and  $q = \#$ negative eigenvalues, and the \emph{signature} is $\sigma(T_A) = p-q$.  Scheja and Storch apply the Stickelberger Trace Formula (Theorem~\ref{ETSS}) to determine the type of $T_A$:

\begin{theorem}[Theorem 94.7 of \cite{SS}]
\label{SSoverreal}
If $A$ is a finite-dimensional $\R$-algebra, then the trace form $T_A$ has type $(r_1+r_2,r_2)$, where $r_1$ {\rm(}resp.\ $r_2${\rm)} is the number maximal ideals $ \mathfrak{m} \subseteq A$ with quotient  $A/\mathfrak{m} \simeq \R$ {\rm(}resp.\ $\C${\rm)}.  
\end{theorem}

\begin{proof}
For maximal ideals $\mathfrak{m}_1,\dots, \mathfrak{m}_r$ of $A$ with quotients $L_1,\dots,L_r$, Theorem~\ref{ETSS} implies that 
\begin{equation}
\label{TAformula}
T_A = \sum_{i=1}^r \lambda_i T_{L_i}.
\end{equation}
Using the bases $\{1\}$ of $\R \subseteq \R$ and $\{1,\sqrt{-1}\}$ of $\R \subseteq \C$, one easily computes that 
\[
\text{matrix of } T_{L_i} = \begin{cases} \quad\  (1) & L_i = \R \quad \text{(happens $r_1$ times)}\\ \begin{pmatrix} 2 \!& \phantom{-}0\\ 0 \!&-2\end{pmatrix} &  L_i = \C\quad \text{(happens $r_2$ times)}.\end{cases}
\]
Since $\lambda_i > 0$ for all $i$, \eqref{TAformula} implies that $T_A$ is represented by a diagonal matrix with $r_1+r_2$ positive entries, $r_2$ negative entries, and possibly many zero entries.  The theorem follows.
\end{proof}

Over $\R$, there is a bijective correspondence between symmetric bilinear forms and quadratic forms.  Thus one can speak of the type and signature of a quadratic form.  In what follows, the quadratic form associated to $T_A$ will be denoted $Q_A$, so
\[
Q_A(\alpha) = T_A(\alpha,\alpha) = \mathrm{Tr}(m_{\alpha^2}).
\]

An immediate consequence of Theorem~\ref{SSoverreal} is the following wonderful result about real solutions of a zero-dimensional polynomial system over $\R$.  

\begin{corollary}
\label{SSrealsolsI}
Assume $\langle f_1, \dots, f_s\rangle \subseteq \R[x_1,\dots,x_n]$ is a zero-dimensional ideal.  Set $A = \R[x_1,\dots,x_n]/\langle f_1, \dots, f_s\rangle$ and let
\[
S = \{a \in \R^n \mid  f_1(a) =  \cdots= f_s(a) = 0\}
\]
be the set of real solutions of $f_1 = \cdots = f_s = 0$.  Then the quadratic form $Q_A$ has signature 
\[
\sigma(Q_A) = \#S = \text{the number of real solutions}.
\]
\end{corollary}

\begin{proof}
The maximal ideals $\mathfrak{m}$ of $A$ come in two flavors:\ the $r_1$ maximal ideals with $A/\mathfrak{m} \simeq \R$ correspond to real solutions, hence elements of $S$, and the $r_2$ maximal ideals with $A/\mathfrak{m} \simeq \C$ correspond to complex-conjugate pairs of nonreal solutions.  Thus 
\[
\# S = r_1 = (r_1+r_2) - r_2 = \sigma(Q_A),
\]
where the last equality follows since $Q_A$ has type $(r_1+r_2,r_2)$ by Theorem~\ref{SSoverreal}.
\end{proof}

There is a long history of using quadratic forms to study the number of real solutions, going back to the work of Jacobi, Hermite and Sylvester in the 19th century.  In 1936,  Krein and Naimark wrote a nice survey of these developments.  An English translation of their paper was published in 1981 as \cite{KN}. 

Historically, real positive solutions were prefered (in one variable, negative solutions where called \emph{false roots} by Cardan).  More generally, given $h \in \R[x_1,\dots,x_n]$, one can ask for solutions $a \in \R^n$ of $f_1 = \cdots = f_s = 0$ that satisfy $h(a) > 0$ or $h(a) < 0$.  An easy adaptation of the proofs of Theorem~\ref{SSoverreal} and Corollary~\ref{SSrealsolsI} leads to the following result:

\begin{theorem}
\label{SSrealsolsII}
Assume $\langle f_1, \dots, f_s\rangle \subseteq \R[x_1,\dots,x_n]$ is a zero-dimensional ideal.  Set $A = \R[x_1,\dots,x_n]/\langle f_1, \dots, f_s\rangle$ and let
\[
S = \{a \in \R^n \mid  f_1(a) =  \cdots= f_s(a) = 0\}
\]
be the set of real solutions of $f_1 = \cdots = f_s = 0$.  If $h \in \R[x_1,\dots,x_n]$, then the quadratic form $Q_{A,h}$ defined by
\[
Q_{A,h}(\alpha) = \mathrm{Tr}(m_{\alpha^2h})
\]
 has signature 
\[
\sigma(Q_{A,h}) = \#\{a \in S \mid  h(a) > 0\} -  \#\{a \in S \mid  h(a) < 0\}.
\]
\end{theorem}

\begin{proof}
As in the proof of Theorem~\ref{SSoverreal}, the Stickelberger Trace Formula from Theorem~\ref{ETSS} easily implies 
\[
Q_{A,h} = \sum_{i=1}^r \lambda_i Q_{L_i,h}.
\]
The $r_1$ indices with $L_i \simeq \R$ correspond to elements $a \in S$, and the isomorphism is given by evaluation at $a$.  Thus we can rewrite the above sum as
\[
Q_{A,h} = \sum_{a \in S}  \lambda_i h(a)\hskip1pt Q_\R + \sum_{L_i\simeq\C} \lambda_i Q_{L_i,h}.
\]
The first sum is a quadratic form of signature 
\[
\#\{a \in S \mid  h(a) > 0\} -  \#\{a \in S \mid  h(a) < 0\}.
\]
Hence it suffices to show that $Q_{L_i,h}$ has signature zero when $L_i \simeq \C$.  Such an isomorphism (there are two) is given by evaluation at one of the corresponding pair of complex-conjugate roots of the system.  Call this root $b$ and set $h(b) = u+iv$.  We leave it as an exercise for the reader to show that for the basis $\{1,\sqrt{-1}\}$, the corresponding bilinear form is represented by the symmetric matrix
\[
 \begin{pmatrix} \phantom{-}2u \!& -2v\\ -2v \!&-2u\end{pmatrix},
 \]
which has eigenvalues $\pm2|h(b)|$.  Thus $Q_{L_i,h}$ has signature zero, and we are done.
\end{proof}

This path from Stickelberger to Corollary~\ref{SSrealsolsI} and Theorem~\ref{SSrealsolsII} is lovely but not what happened historically.  Instead, Paul Pedersen \cite{P} and Eberhard Becker \cite{B} discovered these results independently in 1991, with no knowledge at the time of \S94 of Scheja and Storch.  In 1993, Pedersen joined forces with Marie-Fran\c{c}oise Roy and Aviva Szpirglas to write \cite{PRS}, where the authors comment that
\begin{quote}
The structure theory for finite dimensional algebras which we shall present was
first developed by Stickelberger (see [SS 88]).
\end{quote}
While they never say ``Stickelberger's Theorem", this is the first instance I could find of Stickelberger.  Naturally, there is no reference to a paper of his, though the citation to Scheja and Storch is clear.  A year later, in 1994, Becker  and Thorsten W\"ormann published \cite{BW1}, which includes \cite{PRS} in its references.  Thus the link to Stickelberger via Scheja and Storch was established in the literature by 1993.


\section{Conclusion}

We have seen how Stickelberger's 1897 paper influenced Scheja and Storch in 1988.  His name and the link to Scheja and Storch appeared in papers on real solutions starting in 1993, and in 1995, we finally see the label \emph{Stickelberger's Theorem} applied to the Eigenvalue Theorem, with the name becoming standard in the late 1990s.  But in the process, the link to Stickelberger's actual work got lost.  The purpose of this paper is to reestablish the connection and get a better sense of Stickelberger's contribution.

One thing to notice in the papers from the 1990s is the emphasis on \emph{structure}.  In 1993, Pedersen, Roy and Szpirglas \cite{PRS} use a structure theory for finite dimensional algebras ``first developed by Stickelberger'', and in 1995, Gonzalez-Vega and Trujillo \cite{GVT} state a ``Stickelberger Theorem'' that describes the structure of multiplication matrices.  This is not what Stickelberger did; rather, in \cite{S1897}, he proved a trace formula using the known factorization $p\mathcal{O} = \mathfrak{p}^{e_1}_1\cdots \mathfrak{p}_m^{e_m}$.  

The emphasis on structure is really due to Scheja and Storch in their version of Stickelberger's Trace Formula in Lemma 94.6 in \cite{SS}:\  the  ``fine structure of the trace form'' is a consequence of ``simple features of the algebra itself".  This is borne out by their proof of the lemma.  However, their treatment is abstract and non-constructive, while the papers from the 1990s are interested in algorithms.  In these papers, the goal is not to \emph{describe} the structure  but rather  to \emph{compute} the structure.  This is a significant advance beyond what Stickelberger, Scheja and Storch did.

A striking feature of this story is wide range of mathematics involved:
\begin{itemize}
\item Abstract algebra: G\"unter Scheja and Uwe Storch.
\item Algebraic number theory: Ludwig Stickelberger.
\item Computer algebra: Daniel Lazard and Paul Pedersen.
\item Numerical analysis: Winfried Auzinger and Hans Stetter.
\item Real algebraic geometry: Eberhard Becker, Marie-Fran\c{c}oise Roy, Aviva Szpirglas and Thorsten W\"ormann.
\end{itemize}
Of course, the names mentioned here are involved in other areas of research; what the list represents is the perspectives they brought to the story of Stickelberger and the Eigenvalue Theorem.  

I draw two lessons from this diversity.  First, polynomial systems have a wide interest that touches on many areas of mathematics, and second, abstract algebra provides a powerful language that enables us to understand the structure of what is going on.  As an algebraic geometer, I find this to be deeply satisfying.  

A final comment is that the linear maps $m_f$ commute since $m_f \circ m_g = m_{fg}$.  This point of view features in the version of the Eigenvalue Theorem, valid over an arbitrary field, that appears in \cite[Theorem 2.4.3]{KR}.  See also \cite[Section 6.2.A]{KR}.

As noted at the beginning of Section 3, my account of Stickelberger and the Eigenvalue Theorem omits many fine papers.   I apologize for any omissions or inaccuracies on my part.  

\section*{Acknowledgements} I would like to thank Eberhard Becker, Hubert Flenner, Laureano Gonzalez-Vega, Trevor Hyde, Stefan Kekebus, Michael M\"oller, Markus Reineke, Lorenzo Robbiano, Marie-Fran\c{c}oise Roy and Hans Stetter for helpful correspondence about Stickelberger and the Eigenvalue Theorem.  I am grateful to David Eisenbud for all he has done for mathematics and society.   

I also have a story to tell.  In the late 1970s, I learned about the Eisenbud-Levine Theorem \cite{EL}, which computes the topological degree of a $C^\infty$ map germ as the signature of a certain algebraically defined quadratic form (see \cite{Eisenbud} for a lovely exposition).  The topological degree is usually defined using singular cohomology.  I was working on \'etale cohomology at the time, and my paper \cite{Cox} studied a question raised by David of whether \'etale cohomology can be used to define the topological degree (it can't).  In reading the paper of Eisenbud and Levine, I learned about a splendid 1975 paper of Scheja and Storch \cite{SS75} on Spurfunktionen (trace functions).  In their algebra book \cite{SS}, written thirteen years later, \S94 is entitled \emph{Die Spurformen} (Trace Forms).  This is where they make the connection between Stickelberger and the ideas behind the Eigenvalue Theorem.  So dedicating this paper to David is wonderfully appropriate.

\end{document}